\documentclass[reqno, 11pt]{amsart}
\usepackage{amssymb,latexsym}
\usepackage{amsmath}
\usepackage{graphicx}
\usepackage{amscd}
\usepackage{color}
\usepackage{enumerate}
\newenvironment{enumeratei}{\begin{enumerate}[\upshape (i)]}{\end{enumerate}}
\numberwithin{equation}{section}
\theoremstyle{plain}
 \newtheorem{theorem}{Theorem}[section]
 \newtheorem{lemma}[theorem]{Lemma}
 \newtheorem{proposition}[theorem]{Proposition}
 
\theoremstyle{definition}
 \newtheorem{remark}[theorem]{Remark}
\newcommand \Con{\textup{Con}}
\newcommand \Sizes{\textup{ConSizes}}
\newcommand \Hered{\textup{Hered}}
\newcommand \Dir{\textup{Irr}}
\newcommand \Jir{\textup{J}}
\newcommand \Mir{\textup{M}}
\newcommand \Jred{\textup{JRed}}
\newcommand \Mred{\textup{MRed}}
\newcommand \Qu{\textup{Q}}
\newcommand \con{\textup{con}}
\newcommand \ucov [1] {#1^+}
\newcommand \lcov [1] {#1^-}
\newcommand \cgeq {\equiv_{\textup{con}}}
\newcommand \cgleq {\leq_{\textup{con}}}
\newcommand \nnul {\mathbb N_0}
\newcommand \nplu {\mathbb N^+}
\newcommand \KRlist {\mathcal L_{\textup{KR}}}
\newcommand \jl {\mathop{\vee_{\kern -1pt L}}}
\newcommand \jk {\mathop{\vee_{\kern -1pt K}}}
\newcommand \ml {\mathop{\wedge_L}}

\newcommand \txtleq [1] {\mathrel{\overset{\textup{#1}}\leq}}

\newcommand \tdual[1] {#1$^{\textup d}$}
\newcommand \mdual[1] {#1^{\textup{dual}}}
\newcommand \tdn {\mathord{\searrow}}
\newcommand \tup {\mathord{\nearrow}}
\newcommand \rskip {\kern -8pt}
\newcommand \tuple [1] {\langle #1\rangle}
\newcommand \pair [2] {\tuple{#1,#2}}
\newcommand \tbf [1] {\textbf{#1}} 
\newcommand \set [1] {\{#1\}}
\newcommand \red [1] {{\textcolor{red}{#1}}}

\newcommand \url[1]{{\small{\texttt{#1}}}}
\newcommand \nonau {\\}
\begin{document}
\title[Lattices with many congruences]
{Lattices with many congruences are planar}

\author[G.\ Cz\'edli]{G\'abor Cz\'edli}
\email{czedli@math.u-szeged.hu}
\urladdr{http://www.math.u-szeged.hu/~czedli/}
\address{University of Szeged, Bolyai Institute, 
Szeged, Aradi v\'ertan\'uk tere 1, 
HUNGARY 6720}

\thanks{This research was supported by the Hungarian Research Grant KH 126581}

\subjclass[2010]{06B10\hfill\red{version July 22, 2018}} 


\dedicatory{Dedicated to the memory of Ivan Rival}

\begin{abstract} Let $L$ be an $n$-element finite lattice. We prove that if $L$ has strictly more than $2^{n-5}$ congruences, then $L$ is planar. This result is sharp, since for each natural number $n\geq 8$, there exists a non-planar lattice with 
exactly $2^{n-5}$ congruences. 
\end{abstract}

\maketitle

\section{Aim and introduction}

Our goal is to prove the following statement.

\begin{theorem}\label{thmmain}
Let $L$ be an $n$-element finite lattice. If $L$ has strictly more than $2^{n-5}$ congruences, then it is a planar lattice.
\end{theorem}

In order to point out that this result is sharp, we will also prove the following easy remark. An $n$-element finite lattice $L$ is \emph{dismantlable} if there is a sequence $L_1\subset L_2\subset \dots\subset L_n=L$ of its sublattices such that $|L_i|=i$ for every $i\in\set{1,\dots,n}$; see Rival~\cite{rivaldismentlable}. We know from Kelly and Rival~\cite{kellyrival} that every finite planar lattice is dismantlable.

\begin{remark}\label{remrPSTk}
For each natural number $n\geq 8$, there exists an $n$-element non-dismantlable lattice $L(n)$ with exactly $2^{n-5}$ congruences;
this $L(n)$ is non-planar. 
\end{remark}

We know from Freese~\cite{freesecomplat} that an $n$-element lattice $L$ has at most $2^{n-1}=16\cdot 2^{n-5}$ congruences. In other words, denoting the lattice of congruences of $L$ by $\Con(L)$, we have that $|\Con(L)|\leq 2^{n-1}$. 
For $n\geq 5$, the second largest number of the set 
\[\Sizes(n):=\{|\Con(L)|: L\text{ is a lattice with }|L|=n\}
\]
is $8\cdot 2^{n-5}$ by Cz\'edli~\cite{czglatmanycongr}, while 
Kulin and Mure\c san~\cite{kulinmuresan} proved that the  third, fourth, and fifth largest numbers of $\Sizes(n)$ are 
$5\cdot 2^{n-5}$, $4\cdot 2^{n-5}$, and $\frac 7 2 \cdot 2^{n-5}$, respectively. Since both \cite{czglatmanycongr} and Kulin and Mure\c san~\cite{kulinmuresan} described the lattices witnessing these numbers, 
it follows from these two papers that $|\Con(L)|\geq \frac 7 2 \cdot 2^{n-5}$ implies the planarity of $L$. So, \cite{czglatmanycongr},  \cite{kulinmuresan}, and even their precursor, Mure\c san~\cite{muresan} have naturally lead to the conjecture that if an $n$-element lattice $L$ has \emph{many} congruences with respect to $n$, then $L$ is necessarily planar. However, the present paper needs a technique different from Kulin and Mure\c san~\cite{kulinmuresan}, because a \cite{kulinmuresan}-like description of the 
lattices witnessing the sixth, seventh, eighth, \dots, $k$-th largest numbers in $\Sizes(n)$ seems to be hard to find and  \emph{prove}; we do not even know how large is $k$. 
Fortunately, we can rely on the powerful description of planar lattices 
given by Kelly and Rival~\cite{kellyrival}.

Note that although an $n$-element finite lattice with ``\emph{many}'' (that is, more than $2^{n-5}$) congruences
is necessarily planar by  Theorem~\ref{thmmain}, an $n$-element planar lattice may have only very \emph{few} congruences even for large $n$.  For example, the $n$-element modular lattice of length 2, denoted usually by $M_{n-2}$, has only two congruences if $n\geq 5$. On the other hand, 
we know, say, from Kulin and Mure\c san~\cite{kulinmuresan}  that  there are a lot of lattices $L$ with many congruences, whereby a lot of lattices belong to the scope of Theorem~\ref{thmmain}.

\subsection*{Outline and prerequisites} 
Section~\ref{sectionsomeknown} recalls some known facts from the literature and, based on these facts, proves Remark~\ref{remrPSTk} in three lines.
The rest of the paper is devoted to the proof of Theorem~\ref{thmmain}. 

Due to Section~\ref{sectionsomeknown}, the reader is assumed to have  only little familiarity with lattices. Apart from some figures from Kelly and Rival~\cite{kellyrival}, which should be at hand while reading, the present paper is more or less self-contained modulo the above-mentioned familiarity.  
Note that \cite{kellyrival} is an open access paper at the time of this writing; see  
\url{http://dx.doi.org/10.4153/CJM-1975-074-0}.

\section{Some known facts about lattices and their congruences}\label{sectionsomeknown}
In the whole paper, \emph{all lattices are assumed to be finite} even if this is not repeated all the time.
 For a finite lattice $L$, the set of nonzero join-irreducible elements, that of nonunit meet-irreducible elements, and that of doubly irreducible (neither 0, not 1) elements will be denoted by $\Jir(L)$, $\Mir(L)$, and $\Dir(L)=\Jir(L)\cap\Mir(L)$, respectively. 
For $a\in \Jir(L)$ and $b\in \Mir(L)$, the unique lower cover of $a$ and the unique (upper) cover of $b$ will be denoted by $\lcov a$ and $\ucov b$, respectively. For $a,b\in L$, let $\con(a,b)$ stand for the smallest congruence of $L$ such that $\pair a b\in \con (a,b)$. For $x,y\in \Jir(L)$, let $x \cgeq y$ mean that $\con(\lcov x,x)=\con(\lcov y,y)$. Then $\cgeq$ is an equivalence relation on $\Jir(L)$, and the corresponding quotient set will be denoted by 
\begin{equation} 
\Qu(L) :=  \Jir(L) / \mathord{\cgeq} .
\label{eqQdFnlM}
\end{equation}
As an obvious consequence of Freese, Je\v zek and Nation~\cite[Theorem 2.35]{freesejezeknation} or Nation~\cite[Corollary to Theorem 10.5]{nation},  for every finite lattice $L$, 
\begin{equation}
|\Con(L)| \leq 2^{|\Qu(L)|} \leq 2^{|\Jir(L)|};
\label{eqFJNinequality}
\end{equation}
more explanation will be given later.
The situation simplifies for distributive lattices; 
it is well known that
\begin{equation}
\text{if $L$ is a finite \emph{distributive} lattice, then 
$|\Con(L)| = 2^{|\Jir(L)|}$.}
\label{eqtxtmodBoole}
\end{equation}
In order to explain how to extract \eqref{eqFJNinequality} and \eqref{eqtxtmodBoole} from the literature, we recall some facts.  A \emph{quasiordered set} is a structure $\tuple{A;\leq}$, where $\leq$ is a quasiordering, that is, a reflexive and symmetric relation on $A$.
For example, if we let $a\cgleq b$ mean $\con(\lcov a,a)\leq \con(\lcov b,b)$, then $\tuple{\Jir(L);\cgleq}$ is a quasiordered set.  
A subset $X$ of $\tuple{A;\leq}$ is \emph{hereditary}, if $(\forall x\in X)(\forall y\in A)(y\leq x\Rightarrow y\in X)$. The set of all hereditary subsets of $\tuple{A;\leq}$ with respect to set inclusion forms a lattice 
$\Hered(\tuple{A;\leq})$. Freese, Je\v zek and Nation~\cite[Theorem 2.35]{freesejezeknation} can be reworded as $\tuple{\Con(L);\subseteq}\cong \Hered(\tuple{\Jir(L); \cgleq})$.
To recall this theorem in a form closer to ~\cite[Theorem 2.35]{freesejezeknation}, for the  $\cgeq$-blocks of   $a,b\in \Jir(L)$, we define the meaning of  
$a/\mathord{\cgeq} \cgleq  a/\mathord{\cgeq}$ as $a\cgleq b$. In this way, we obtain a poset $\tuple{\Qu(L); \cgleq/ \mathord{\cgeq} }$.  With this notation, the original form of  \cite[Theorem 2.35]{freesejezeknation} states that $\Con(L)\cong \Hered(\Qu(L); \cgleq/ \mathord{\cgeq} )$.

Since $\cgeq$ will play an important role later, recall that for intervals $[a,b]$ and $[c,d]$ in a lattice $L$, $[a,b]$ \emph{transposes up} to $[c,d]$ if $b\wedge c=a$ and $b\vee c=d$. This relation between the two intervals will be denoted by $[a,b]\tup [c,d]$.  We say that  $[a,b]$ \emph{transposes down} to $[c,d]$, in notation   $[a,b]\tdn [c,d]$ if $[c,d]\tup [a,b]$. We call $[a,b]$ and $[c,d]$ \emph{transposed intervals} if $[a,b]\tdn [c,d]$ or  $[a,b]\tup [c,d]$. 
It is well known and easy to see that 
\begin{equation}
\text{if  $[a,b]$ and $[c,d]$ are transposed intervals, then $\con(a, b)=\con(c, d)$.}
\label{eqtrpnlDkszgcsprGm}
\end{equation}

Next, we note that $\Con(L)$ in \eqref{eqtxtmodBoole} is a Boolean lattice. Oddly enough, we could find neither this fact, nor  \eqref{eqtxtmodBoole} in the literature explicitly. Hence, in this paragraph, we outline briefly a possible way of deriving these statements from
explicitly available and well-known other facts; the reader may skip over this paragraph.
So let $L$ be a finite \emph{distributive} lattice. 
Since $L$ is modular,   $\Con(L)$ is a Boolean lattice by  Gr\"atzer~\cite[Corollary 3.12 in page 41]{grbypict}.
Pick a maximal chain $0\prec a_1\prec\dots\prec a_t=1$ in $L$. 
Here $t$ is the length of $L$, and it is well known that $t=|\Jir(L)|$; see Gr\"atzer~\cite[Corollary 112 in page 114]{ggglt}.
If $\con(a_{i-1},a_{i})=\con(a_{j-1},a_{j})$, then it follows 
from Gr\"atzer~\cite{ggprimeprojective} and distributivity (in fact, modularity)
that there is a sequence  of prime intervals (edges in the diagram) from $[a_{i-1},a_{i}]$ to $[a_{j-1},a_{j}]$ such that any two neighboring intervals in this sequence are transposed. In the terminology of Adaricheva and Cz\'edli~\cite{adarichevaczedli}, the prime intervals $[a_{i-1},a_{i}]$ and $[a_{j-1},a_{j}]$ belong to the same \emph{trajectory}. Since no two distinct comparable prime intervals of $L$ can belong to the same trajectory by 
\cite[Proposition 6.1]{adarichevaczedli}, it follows that $i=j$. 
Hence, the congruences $\con(a_{i-1},a_{i})$, $i\in\set{1,\dots, t}$, are pairwise distinct. They are join-irreducible congruences by Gr\"atzer~\cite[page 213]{ggglt}, whereby they are atoms in $\Con(L)$ since $\Con(L)$ is Boolean. Clearly, $\bigvee_{i=1}^t \con(a_{i-1},a_{i})$ is $1_{\Con(L)}$, which implies that $|\Con(L)|=2^t$. This proves  \eqref{eqtxtmodBoole} since $t=|\Jir(L)|$.

Next, a lattice is called \emph{planar} if it is finite and has a Hasse-diagram that is a planar representation of a graph in the usual sense that any two edges can intersect only at vertices. Let $\nnul$ and $\nplu$ denote the set
$\set{0,1,2,\dots}$ of nonnegative integers and  the set $\set{1,2,3,\dots}$ of positive integers, respectively. In their fundamental paper on planar lattices, Kelly and Rival~\cite{kellyrival} gave a set
\[
\KRlist=\set{A_n, E_n,F_n,G_n,H_n:n\in\nnul}\cup\set{B,C,D}
\]
of finite lattices such that the following statement holds.

\begin{proposition}[{A part of Kelly and Rival~\cite[Theorem~1]{kellyrival}}]
\label{probKRthm}
A finite lattice $L$ is planar if and only if neither $L$, nor its dual  contains some lattice of $\KRlist$ as a subposet.
\end{proposition}

Note that the lattices $A_n$ are selfdual and  Kelly and Rival~\cite{kellyrival} proved the minimality of $\KRlist$, but we do not need these facts.

Next, we prove Remark~\ref{remrPSTk}. The \emph{ordinal sum} of lattices $L'$ and $L''$ is their disjoint union $L'\mathop{\dot\cup} L''$ such that 
for $x,y\in L'\mathop{\dot\cup} L''$, we have that $x\leq y$ if and only if $x\leq_{L'}y$, or $x\leq_{L''}y$, or $x\in L'$ and $y\in L''$. 

\begin{proof}[Proof of Remark~\ref{remrPSTk}] Let $L(8)$ be the eight-element Boolean lattice. Next, for $n>8$, let $L(n)$ be the ordinal sum of $L(8)$ and an $(n-8)$-element chain. 
Since $|\Jir(L_n)|=n-5$,  Remark~\ref{remrPSTk} follows from \eqref{eqtxtmodBoole}.
\end{proof}

Note that $L(n)$ above occurs also in page 93 of Rival~\cite{rivaldismentlable}.

\section{A lemma on subposets that are lattices}
While $\KRlist$ consists of \emph{lattices}, they appear in Proposition~\ref{probKRthm} as \emph{subposets}.  This fact causes some difficulties in proving our theorem; this section serves as a preparation to overcome these difficulties. The set of \emph{join-reducible elements} of a lattice $L$ will be denoted by $\Jred(L)$. Note that 
\begin{equation}
\Jred(L)=L\setminus(\set 0\cup\Jir(L))=\set{a\vee b: a,b\in L\text{ and }a\parallel b},
\label{eqJrdlsmwGjT}
\end{equation}
where $\parallel$ stands for incomparability, that is, $a\parallel b$ is the conjunction of $a\nleq b$ and $b\nleq a$. Similarly, $\Mred(L)=L\setminus(\set 1\cup\Mir(L))$ denotes the set of \emph{meet-reducible elements} of $L$.

\begin{lemma}\label{lemmalatsubPoset}
Let $L$ and $K$ be finite lattices such that $K$ is a subposet of $L$. Then the following four statements and their duals hold.
\begin{enumeratei}
\item\label{lemmalatsubPoseta} If $a_1,\dots,a_t\in K$ and $t\in\nplu$, then $a_1\jl\dots \jl a_t\leq a_1\jk\dots \jk a_t$. 
\item\label{lemmalatsubPosetb} If $t,s\in\nplu$,  $a_1,\dots,a_t, b_1,\dots,b_s\in K$, and $a_1\jk\dots \jk a_t$ is distinct from $b_1\jk\dots \jk b_s$, then  $a_1\jl\dots \jl a_t\neq b_1\jl\dots \jl b_s$.
\item\label{lemmalatsubPosetc} $|\Jred(L)|\geq |\Jred(K)|$ and, dually, $|\Mred(L)|\geq |\Mred(K)|$.
\item\label{lemmalatsubPosetd} If  $|\Jred(L)| = |\Jred(K)|$, $u_1,u_2,v_1,v_2\in K$, $u_1\parallel u_2$, $v_1\parallel v_2$, and $u_1\jk u_2 = v_1\jk v_2$, then  $u_1\jl u_2 = v_1\jl v_2$.
\end{enumeratei}
\end{lemma}

Note that, according to  \eqref{lemmalatsubPosetb} and  \eqref{lemmalatsubPosetd}, the distinctness of joins is generally preserved when passing from $K$ to $L$, but equalities are preserved only under additional assumptions. The dual of a condition (X) will be denoted by \tdual{(X)};
for example, the dual of Lemma~\ref{lemmalatsubPoset}\eqref{lemmalatsubPoseta}
is denoted by   Lemma~\ref{lemmalatsubPoset}\tdual{\eqref{lemmalatsubPoseta}} or simply by \ref{lemmalatsubPoset}\tdual{\eqref{lemmalatsubPoseta}}.

\begin{proof}[Proof of Lemma~\ref{lemmalatsubPoset}] Part \eqref{lemmalatsubPoseta} is a trivial consequence of the concept of joins as least upper bounds.

In order to prove \eqref{lemmalatsubPosetb}, assume that 
$a_1\jl\dots \jl a_t = b_1\jl\dots \jl b_s$.
Part \eqref{lemmalatsubPoseta}
gives that $a_i\leq_L a_1\jl\dots \jl a_t = b_1\jl\dots \jl b_s\leq_L b_1\jk\dots \jk b_s$, for all $i\in\set{1,\dots,t}$. Since $K$ is a subposet of $L$, $a_i\leq_K b_1\jk\dots \jk b_s$. But $i\in\set{1,\dots,t}$ is arbitrary, whereby $a_1\jk\dots \jk a_t\leq_K b_1\jk\dots \jk b_s$. We have equality here, since the converse inequality follows in the same way. Thus, we conclude \eqref{lemmalatsubPosetb} by contraposition.

Next, let $\set{c_1,\dots, c_t}$ be a repetition-free list of $\Jred(K)$. For each $i$ in $\set{1,\dots,t}$, pick $a_i,b_i\in K$ such that $a_i\parallel b_i$ and $c_i=a_i\jk b_i$. That is,
\begin{equation}
\Jred(K)=\set{c_1=a_1\jk b_1, \dots,  c_t=a_t\jk b_t}.
\label{eqrNfGhtlTrDrS}
\end{equation}
Since $a_i\parallel b_i$ holds also in $L$,
\begin{equation}
\set{a_1\jl b_1, \dots,  a_t\jl b_t}\subseteq \Jred(L).
\label{eqkjpKkPlDsznORscSj}
\end{equation}
The elements listed in \eqref{eqkjpKkPlDsznORscSj} are pairwise distinct by part \eqref{lemmalatsubPosetb}.
Therefore, $|\Jred(K)|=t\leq |\Jred(L)|$, proving part \eqref{lemmalatsubPosetc}.

Finally, to prove part \eqref{lemmalatsubPosetd}, we assume the premise of part \eqref{lemmalatsubPosetd}, and we let $t:=|\Jred(K)| = |\Jred(L)|$. Choose $c_i,a_i,b_i\in K$ as in \eqref{eqrNfGhtlTrDrS}. 
Since $t=|\Jred(L)|$,  part  \eqref{lemmalatsubPosetb}  and \eqref{eqkjpKkPlDsznORscSj} give that 
\begin{equation}
\Jred(L)=\set{a_1\jl b_1, \dots,  a_t\jl b_t}.
\label{eqZhfTrNwmS}
\end{equation}
As a part of the premise of \eqref{lemmalatsubPosetd},
 $u_1\parallel u_2$ has been assumed. Hence, 
\eqref{eqrNfGhtlTrDrS} yields  a unique subscript $i\in\set{1,\dots,t}$ such that 
$c_i=u_1\jk u_2 = v_1\jk v_2$. Since $c_1,\dots, c_t$ is a repetition-free list of the elements of $\Jred(K)$, we have that 
$u_1\jk u_2 \neq c_j=a_j\jk b_j$ for  every $j\in \set{1,\dots,t}\setminus\set i$. 
So, for all $j\neq i$, part \eqref{lemmalatsubPosetb} gives that $u_1\jl u_2 \neq a_j\jl b_j$.  But $u_1\jl u_2\in \Jred(L)$, whence 
\eqref{eqZhfTrNwmS} gives that 
$u_1\jl u_2 =a_i\jl b_i$. Since the equality $v_1\jl v_2 =a_i\jl b_i$ follows in the same way, we conclude that $u_1\jl u_2=v_1\jl v_2$, as required. 
This yields part \eqref{lemmalatsubPosetd} and completes the proof of Lemma~\ref{lemmalatsubPoset}.
\end{proof}

Note that  $a_i\jl b_i$ in the proof above can be distinct from $c_i$; this will be exemplified by Figures~\ref{figFnul} and \ref{figEnul}.

\section{The rest of the proof}
In this section, to ease our terminology, let us agree on the following convention. We say that a finite lattice $L$ has \emph{many congruences} if $|\Con(L)|>2^{|L|-5}$. Otherwise, if $|\Con(L)|\leq 2^{|L|-5}$, then we say that $L$ has \emph{few congruences}. 

\begin{lemma}\label{lemmangyvghrmgy} For every  finite lattice $L$, the following two assertions holds.
\begin{enumeratei}
\item\label{lemmangyvghrmgya} If $|\Jred(L)|\geq 4$ or $|\Mred(L)|\geq 4$, then $L$ has few congruences.
\item\label{lemmangyvghrmgyb}  If $|\Jred(L)|=3$ and there are $p,q\in \Jir(L)$ such that $p\neq q$ and $\con(\lcov p,p)=\con(\lcov q,q)$, then $L$ has few congruences.
\end{enumeratei}
\end{lemma}

\begin{proof} Let $n:=|L|$.
If $|\Jred(L)|\geq 4$, then \eqref{eqJrdlsmwGjT} leads to $|\Jir(L)|\leq n-5$, and it follows by \eqref{eqFJNinequality} that $L$ has few congruences. By duality, this proves part~\eqref{lemmangyvghrmgya}. Under the assumptions of  \eqref{lemmangyvghrmgyb}, $p\cgeq q$, and we obtain from \eqref{eqQdFnlM} that 
$|\Qu(L)|\leq |\Jir(L)|-1 = n-4 - 1=n-5$, and \eqref{eqFJNinequality} implies again that $L$ has few congruences. This proves the lemma.
\end{proof}

\begin{figure}[htb]
\centerline{\includegraphics[scale=1.0]{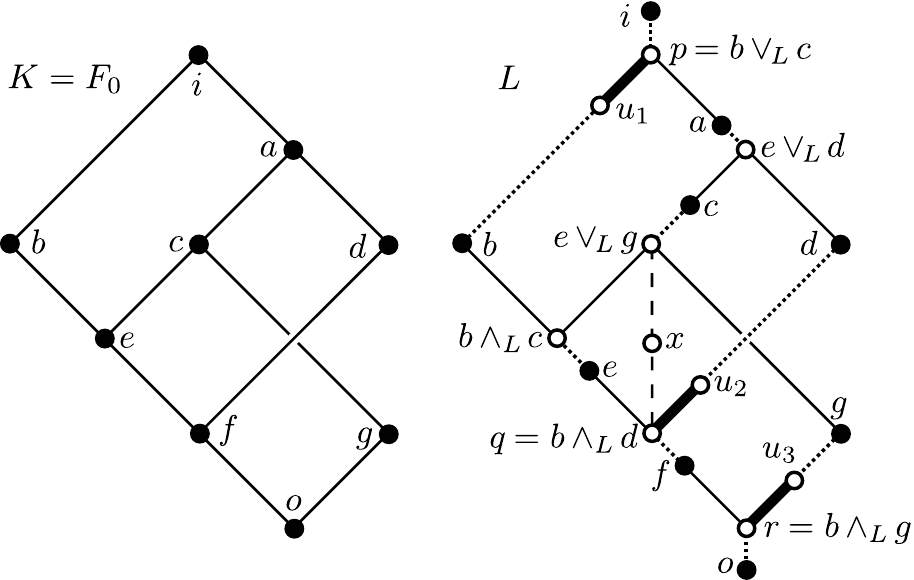}}
\caption{$K=F_0$ and an example for $L$ containing $K$ as a subposet\label{figFnul}}
\end{figure}

\begin{lemma}\label{lemmaFnul} Let $K:=F_0\in\KRlist$, see on the left in  Figure~\ref{figFnul}. 
If $K$ is a subposet of a finite lattice $L$, then $L$ has few congruences.
\end{lemma}

\begin{proof}
Label the elements of $K=F_0$ as shown in Figure~\ref{figFnul}. A possible $L$ is given on the right in Figure~\ref{figFnul}; the elements of $K$ are black-filled. The diagram of $L$ is understood as follows: for  $y_1, y_2\in L$, a \emph{thick solid edge}, a \emph{thin solid edge}, and a \emph{thin dotted edge} ascending from $y_1$ to $y_2$ mean that, in the general case, we know that  $y_1\prec y_2$, $y_1< y_2$, and $y_1\leq y_2$, respectively.  In a concrete situation, further relations can be fulfilled; for example, a thin dotted edge can happen to denote that $y_1=y_2$. 
The two dashed edges and the element $x$ as well as similar edges and elements can be present but they can also be  missing.
Note that $y_1\leq y_2$ is understood as $y_1\leq_L y_2$; for $y_1,y_2\in K$, this is the same as $y_1\leq_K y_2$ since $K$ is a subposet of $L$.
Since $L$ in the figure carries a lot of information on the general case, the reader may choose to inspect Figure~\ref{figFnul} instead of checking some of our computations that will come later. 
Note also that the convention above applies only for $L$; for $K$, every edge stands for covering.

Clearly, $|\Jred(K)|=|\Mred(K)|=3$. Hence,  Lemma~\ref{lemmalatsubPoset}\eqref{lemmalatsubPosetc} gives that  $|\Jred(L)|\geq 3$ and $|\Mred(L)|\geq 3$.
We can assume that none of $|\Jred(L)|\geq 4$ and $|\Mred(L)|\geq 4$ holds, because otherwise Lemma~\ref{lemmangyvghrmgy}\eqref{lemmangyvghrmgya} would immediately complete the proof. Hence, 
\begin{equation}
 |\Jred(L)| = 3\quad\text{and}\quad|\Mred(L)| = 3.
\label{eqtzhGrhBhWxm}
\end{equation}
Since $|\Jred(K)|=|\Mred(K)|=3$ holds also for $K=E_0$, to be  given later in Figure~\ref{figEnul}, note at this point that \eqref{eqtzhGrhBhWxm} will be valid in the proof of Lemma~\ref{lemmaEnul}.
Let $p:=b\jl c\in L$, and let $u_1\in L$ be a lower cover of $p$ in the interval $[b,p]_L$. Also, let $q:=b\ml d$ and let $u_2\in [q,d]_L$ be a cover of $q$. Finally, let $r:=b\ml g$, and let $u_3\in [r,g]_L$ be a cover of $r$. 
Since we have formed the joins and the meets of incomparable elements in $L$ such that the corresponding joins are pairwise distinct in $K$ and the same holds for the meets, \eqref{eqtzhGrhBhWxm} and Lemma~\ref{lemmalatsubPoset} imply that
\begin{equation}
\Jred(L) = \set{p, e\jl d , e\jl g, }\quad\text{and}\quad \Mred(L) = \set{q,r,b\ml c}.
\label{eqpRpzskPkNvl}
\end{equation}
Since $u_1\prec_L p$, $u_1\neq p$. If we had that $u_1= e\jl d$, then 
\[b\leq u_1  =  e\jl d  \txtleq{\ref{lemmalatsubPoset}\eqref{lemmalatsubPoseta}} e\jk d=a
\] 
would contradict $b\nleq_K a$. Replacing $\pair d a$ by $\pair g c$, we obtain similarly that $u_1\neq e\jl g$.  Hence, \eqref{eqpRpzskPkNvl} gives that  $u_1\in \Jir (L)$. If we had that $u_2=p$, then $b\leq p=u_2\leq d$ would be a contradiction. Similarly, $u_2=e\jl d$ would lead to $e\leq e\jl d=u_2\leq d$ while $u_2=e\jl g$ again to $e\leq e\jl g=u_2\leq d$, which are contradictions. Hence, $u_2\notin \Jred(L)$ and so $0_L \leq q \prec_L u_2$ gives that $u_2\in \Jir (L)$.  We have that $u_3\neq p$, because otherwise 
$b\leq p=u_3\leq g$ would be a contradiction. Similarly, 
$u_3=e\jl d$ and $u_3\leq e\jl g$ would lead to the contradictions
$e\leq e\jl d =u_3\leq g$ and $e\leq e\jl g =u_3\leq g$, respectively. 
So, $u_3\notin \Jred(L)$ by \eqref{eqpRpzskPkNvl}. Since $r\prec_L u_3$ excludes
that $u_3=0$, we obtain that $u_3\in \Jir(L)$. 
Since $u_3=u_2$ would lead to 
\begin{equation}
f \txtleq{\ref{lemmalatsubPoset}\tdual{\eqref{lemmalatsubPoseta}} } q \leq u_2=u_3\leq g,
\end{equation}
which is a contradiction, we have that 
\begin{equation}
u_1,u_2,u_3\in \Jir(L)\quad\text{and}\quad u_2\neq u_3.
\label{eqczhGstrHnb}
\end{equation}
%


Next, we claim that 
\begin{equation}
[q,u_2]\tup [u_1,p]\quad\text{and}\quad [u_1,p]\tdn [
r,u_3].
\label{eqbRzspszmkZhgnMt}
\end{equation}
Since  $b\nleq a$, $b\nleq c$,  and $b\nleq d$, none of $e\jl d$, $e\jl g$, and $u_2$ belongs to $[b,i]_L$.  
In particular, we obtain from $u_1\prec p$ and \eqref{eqpRpzskPkNvl} that
\begin{equation}
[b,u_1]_L\subseteq \Jir (L)\quad\text{and}\quad b\nleq u_2. 
\label{eqmdcMrTblTn}
\end{equation}
Suppose, for a contradiction, that $u_2\leq u_1$, and pick a maximal chain in the interval $[u_2,u_1]$. So we pick a lower cover of $u_1$, then a lower cover of the previous lower cover, etc., and it follows from  \eqref{eqmdcMrTblTn} that this chain contains $b$. Hence, $u_2\leq b$, and we obtain that $q\prec_L u_2\leq b\ml d = q$, a contradiction. Hence, $u_2\nleq u_1$. This means that $u_1\ml u_2< u_2$. But $q\leq b\leq u_1$, so we have that 
$q\leq u_1\wedge u_2< u_2$. Since $q\prec u_2$, we obtain that $u_1\ml u_2=q$. 
Similarly, $u_2\leq d\leq p$ and $u_2\nleq u_1$ give that
$u_1<u_1\jl u_2 \leq p$, whereby $u_1\prec_L p$ yields that $u_1\jl u_2=p$. 
The last two equalities imply the first half of \eqref{eqbRzspszmkZhgnMt}. 
The second half follows basically in the same way, so we give less details. 
Based on \eqref{eqmdcMrTblTn}, $u_3\leq u_1$ would lead to $u_3\leq b$ and 
$r\prec_L u_3\leq b\ml g =r$, whence $u_3\nleq u_1$. Since
$u_3\leq  g \leq c \leq b\jl c = p$, $r=b\ml g \leq b\leq u_1$, we obtain that
$r\leq u_1\ml u_3< u_3$ and  $u_1<u_1\jl u_3 \leq p$. Hence the covering relations $r\prec_L u_3$ and  $u_1\prec_L p$ imply the second half of \eqref{eqbRzspszmkZhgnMt}.

Finally, \eqref{eqtrpnlDkszgcsprGm} and  \eqref{eqbRzspszmkZhgnMt} give that 
$\con(q,u_2)=\con(u_1,p)=\con(r,u_3)$. Since  \eqref{eqczhGstrHnb} allows us to replace  $q$ and $r$  by $\lcov u_2$ and $\lcov u_3$, respectively, we obtain that $\con(\lcov u_2,u_2)=\con(\lcov u_3,u_3)$. But $u_2$ and $u_3$ are distinct elements of $\Jir(L)$ by \eqref{eqczhGstrHnb},  whereby \eqref{eqtzhGrhBhWxm} and Lemma~\ref{lemmangyvghrmgy}\eqref{lemmangyvghrmgyb} imply that $L$ has few congruences, as required. This completes the proof of Lemma~\ref{lemmaFnul}.
\end{proof}

We still need another lemma.

\begin{figure}[htb]
\centerline{\includegraphics[scale=1.0]{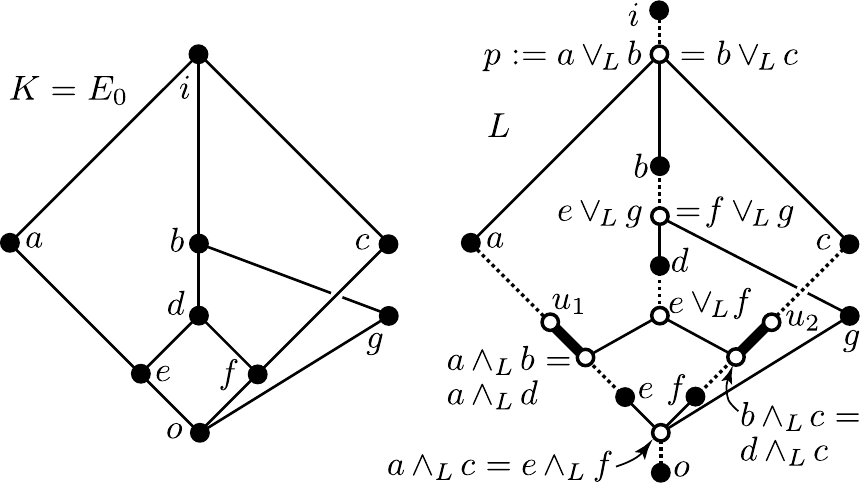}}
\caption{$K=E_0$ and an example for $L$ containing $K$ as a subposet\label{figEnul}}
\end{figure}

\begin{lemma}\label{lemmaEnul} Let $K:=E_0\in\KRlist$, see on the left in  Figure~\ref{figEnul}. 
If $K$ is a subposet of a finite lattice $L$, then $L$ has few congruences.
\end{lemma}

\begin{proof} This proof shows a lot of similarities with the earlier proof. In particular, the same convention applies for the diagram of $L$ in Figure~\ref{figEnul} and, again, there can be several elements of $L$ not indicated in the diagram. 
We have already noted that  \eqref{eqtzhGrhBhWxm} holds in the present situation. Figure~\ref{figEnul} shows how to pick $u_1,u_2\in L$; they are covers of $a\ml b$ in $[a\ml b,a]$ and $b\ml c$ in $[b\ml c,c]$, respectively. As a counterpart of \eqref{eqpRpzskPkNvl}, now we obtain in the same way from 
 \eqref{eqtzhGrhBhWxm} and Lemma~\ref{lemmalatsubPoset} that
\begin{equation}
\left.
\begin{aligned}
\Jred(L) &= \set{p:=a\jl b=b\jl c,\, e\jl f,\, e\jl g=f\jl g}\text{ and}\cr 
\Mred(L) &= \set{a\ml b=a\ml d,\, b\ml c=d\ml c,\, a\ml c=e\ml f}.
\end{aligned}
\rskip\right\}
\label{eqpRpzshTplcsNq}
\end{equation}
Not all the equalities above will be used but they justify Figure~\ref{figEnul}. In particular, even if there can be more elements, the comparabilities and the incomparabilities in the figure are correctly indicated. 
Using \eqref{eqpRpzshTplcsNq} in the same way as we used  \eqref{eqpRpzskPkNvl} in the proof of Lemma~\ref{lemmaFnul} and the above-mentioned correctness of  Figure~\ref{figEnul}, it follows that 
\begin{equation}
[a\ml b,a]\setminus\set{a\ml b}\subseteq \Dir(L),\text{ and }
[a,p]\setminus\set{p}\subseteq \Dir(L).
\label{eqnnSnNszGhK}
\end{equation}
In particular, $u_1\in \Jir(L)$. Similarly to the argument verifying \eqref{eqbRzspszmkZhgnMt}, now \eqref{eqnnSnNszGhK} implies that $[\lcov u_1,u_1]=[a\ml b,u_1]\tup [b,p]$.
Since $\pair a {u_1}$ and $\pair c {u_2}$ play symmetric roles,
we obtain that $u_2\in \Jir(L)$ and  $[\lcov u_2,u_2]=[b\ml c,u_2]\tup [b,p]$. Hence, \eqref{eqtrpnlDkszgcsprGm} gives that 
$\con(\lcov u_1,u_1)=\con(b,p)=\con(\lcov u_2,u_2)$. Since $u_1$ and $u_2$ are distinct by Figure~\ref{figEnul} and they belong to 
$\Jir(L)$ by \eqref{eqnnSnNszGhK} and the $\pair a {u_1}$--$\pair c {u_2}$-symmetry, \eqref{eqtzhGrhBhWxm} and Lemma~\ref{lemmangyvghrmgy}\eqref{lemmangyvghrmgyb} imply that $L$ has few congruences. This completes the proof of Lemma~\ref{lemmaEnul}.
\end{proof}

Now, we are in the position to prove our theorem.

\begin{proof}[Proof of Theorem~\ref{thmmain}]
Let $L$ be an arbitrary non-planar lattice; it suffices to show that $L$ has few congruences. By Proposition~\ref{probKRthm}, there is a lattice $K$ in Kelly and Rival's list $\KRlist$ such that $K$ is a subposet of $L$ or the dual $\mdual L$ of $L$. Since $\Con(\mdual L)=\Con(L)$, we can assume that $K$ is a subposet of $L$. 
A quick glance at the lattices of $\KRlist$, see their diagrams in Kelly and Rival~\cite{kellyrival}, shows that if $K\in\KRlist\setminus\set{E_0,F_0}$, then $|\Jir(K)|\geq 4$ or $|\Mir(K)|\geq 4$. Hence, if $K\in\KRlist\setminus\set{E_0,F_0}$, then Lemma~\ref{lemmangyvghrmgy}\eqref{lemmangyvghrmgya} implies that $L$ has few congruences, as required. If $K\in \set{E_0,F_0}$, then the same conclusion is obtained by Lemmas~\ref{lemmaFnul} and \ref{lemmaEnul}. This completes the proof of Theorem~\ref{thmmain}.
\end{proof}

\end{document}